\documentclass[10pt]{article}
\usepackage{amssymb,latexsym,amscd, amsthm, amsmath, enumerate, subfigure}

\newtheorem{Thm}{Theorem}[section]

\newtheorem{Ex}[Thm]{Example}
\newtheorem{Rmk}[Thm]{Remark}

\newtheorem{Conj}[Thm]{Conjecture}
\date{}

\newcommand{\QQ}{{\mathbb{Q}}}
\newcommand{\RR}{{\mathbb{R}}}

\newcommand{\val}{\text{val}}
\newcommand{\GT}{{\mathcal{G}_{2,n}}}
\newcommand{\TN}{{\mathcal{T}_{n}}}

\newcommand{\TT}{{\mathcal{T}}}
\newcommand{\trop}{{\text{trop}}}

\begin{document}

\title{\textbf{On the relation between weighted trees and tropical Grassmannians}}
\author{Filip Cools \footnote{K.U.Leuven, Department of Mathematics,
Celestijnenlaan 200B, B-3001 Leuven, Belgium, email:
Filip.Cools@wis.kuleuven.be; the author is a postdoctoral fellow of the Research Foundation - Flanders (FWO)}}

\maketitle {\footnotesize \emph{\textbf{Abstract.---} In this article, we will prove that the set of $4$-dissimilarity vectors of $n$-trees is contained in the tropical Grassmannian $\mathcal{G}_{4,n}$. We will also propose three equivalent conjectures related to the set of $m$-dissimilarity vectors of $n$-trees for the case $m\geq 5$. Using a computer algebra system, we can prove these conjectures for $m=5$.}}
\\ ${}$

\section{Introduction}

Let $T$ be a tree with $n$ leaves, which are numbered by the set $[n]:=\{1,\ldots,n\}$. Such a tree is called an $n$-tree. We assume that $T$ is weighted, so each edge has a length.  Denote by $D(i,j)$ the distance between the leaves $i$ and $j$ (i.e. the sum of the lengths of the edges of the unique path in $T$ from $i$ to $j$). We say that $D=(D(i,j))_{i,j}\in \RR^{n\times n}$ is the dissimilarity matrix of $T$, or conversely, that $D$ is realized by $T$. The set of dissimilarity matrices of $n$-trees is fully described by the following theorem (see \cite{Bun} or \cite[Theorem 2.36]{PaSt}).

\begin{Thm}[Tree Metric Theorem] \label{thm tree metric}
Let $D\in \RR^{n\times n}$ be a symmetric matrix with zero entries on the main diagonal. Then $D$ is a dissimilarity matrix of an $n$-tree if and only if the \textit{four-point condition} holds, i.e. for every four (not necessarily distinct) elements $i,j,k,l\in [n]$, the maximum
of the three numbers $D(i,j)+D(k,l)$, $D(i,k)+D(j,l)$ and
$D(i,l)+D(j,k)$ is attained at least twice. Moreover, the $n$-tree $T$ that realizes $D$ is unique.
\end{Thm}

If $T$ is an $n$-tree, $(D(i,j))_{i<j}\in \RR^{n\choose 2}$ is called the dissimilarity vector of $T$.

We can reformulate the above theorem in the context of tropical geometry (see \cite[Theorem 4.2]{SS1}). For some background, I refer to section \ref{section tropical}.

\begin{Thm}\label{thm space trees} The set $\TN$ of dissimilarity vectors of $n$-trees is equal to the tropical Grassmannian $\mathcal{G}_{2,n}$. \end{Thm}

We can generalize the definition of dissimilarity vectors of $n$-trees. Let $m$ be an integer with $2\leq m<n$ and let $i_1,\ldots,i_m$ be pairwise distinct elements of $\{1,\ldots,n\}$. Denote by $D(i_1,\ldots,i_m)$ the length of the smallest subtree of $T$ containing the leaves $i_1,\ldots,i_m$. We say that the point $D=(D(i_1,\ldots,i_m))_{i_1<\ldots<i_m}\in \RR^{n\choose m}$ is the $m$-dissimilarity vector of $T$.

The following result gives a formula for computing the $m$-subtree weights from the pairwise distances of the leafs of an $n$-tree (see \cite[Theorem 3.2]{BoCo}).

\begin{Thm} \label{thm phi}
Let $n$ and $m$ be integers such that $2\leq m<n$. Denote by $\mathcal{C}_m\subset \mathcal{S}_m$ the set of cyclic permutations of length $m$. Let
$$\phi^{(m)}:\RR^{n \choose 2}\to \RR^{n \choose m}:
X=(X_{i,j})\mapsto (X_{i_1,\ldots,i_m})$$ be the map with
$$
X_{i_1,\ldots,i_m}=\frac12 \cdot \min_{\sigma \in \mathcal{C}_m}
\{X_{i_1,i_{\sigma(1)}}+X_{i_{\sigma(1)},i_{\sigma^2(1)}}+\ldots+X_{i_{\sigma^{m-1}(1)},i_{\sigma^m(1)}}\}.
$$
If $D\in \TN\subset \mathbb{R}^{n \choose 2}$ is the dissimilarity vector of an $n$-tree $T$, then
the $m$-dissimilarity vector of $T$ is equal to $\phi^{(m)}(D)$. So $\phi^{(m)}(\TN)$ is the set of $m$-dissimilarity vectors of $n$-trees.
\end{Thm}

The description of the set of $m$-dissimilarity vectors of $n$-trees as the image of $\TN$ under the map $\phi^{(m)}$ is not useful to decide wether or not a given point in $\RR^{n\choose m}$ is an $m$-dissimilarity vector. So we are interested in finding a nice description of these sets as subsets of $\RR^{n\choose m}$. The case $m=3$ is solved by the following result (see \cite[Theorem 4.6]{BoCo}).

\begin{Thm} $\phi^{(3)}(\TN)=\mathcal{G}_{3,n}\cap \phi^{(3)}(\RR^{n\choose 2})$. \end{Thm}

In this article, we prove the following partial answer for the case $m=4$.

\begin{Thm} \label{thm main} $\phi^{(4)}(\TN)\subset \mathcal{G}_{4,n}\cap \phi^{(4)}(\RR^{n\choose 2})$. \end{Thm}

To finish the article, we propose three equivalent conjectures for the case $m\geq 5$. The case $m=5$ is solved using a computer algebra system.

\section{Tropical geometry} \label{section tropical}

Consider the tropical semi-ring $(\mathbb{R}\cup\{-\infty\},\oplus,\otimes)$, where the tropical sum is the maximum of two numbers and the tropical product is the usual sum of the numbers. Let $x_1,\ldots,x_k$ be real variables. Tropical monomials $x_1^{i_1}\ldots x_k^{i_k}$ represent linear forms $i_1 x_1+\ldots+i_k x_k$ and tropical polynomials $\oplus_{i\in I} a_i x_1^{i_1}\ldots x_k^{i_k}$ (with $I\subset \mathbb{N}^k$ finite) represent piece-wise linear forms \begin{equation}\max_{i\in I}\{a_i+i_1 x_1+\ldots+i_k x_k\}.\label{troppoly}\end{equation} If $F$ is such a tropical polynomial, we define the tropical hypersurface $\mathcal{H}(F)$ to be its corner locus, i.e. the points $x\in \RR^k$ where the maximum is attained at least twice.

Let $K=\mathbb{C}\{\{t\}\}$ be the field of Puiseux series, i.e. the field of formal sums $c=\sum_{q\in \mathbb{Q}} c_q t^q$ in the variable $t$ such that the set $S_c=\{q|c_q\neq 0\}$ is bounded below and has a finite set of denominators. For each $c\in K^{\ast}$, the set $S_c$ has a minimum, which we call the valuation of $c$ and is denoted by $\val(c)$.

A polynomial $f=\sum_{i\in I} f_i x_i^{i_1}\ldots x_k^{i_k}$ over $K$ gives rise to a tropical polynomial $\trop(f)$, defined by taking $a_i=-\val(f_i)$ in \eqref{troppoly}.

\begin{Thm} \label{Thm def tropical variety}
If $I\subset K[x_1,\dots,x_k]$ is an ideal, the following two subsets of $\RR^k$ coincide:
\begin{enumerate}
\item the intersection of all tropical hypersurfaces $\TT(\trop(f))$ with $f\in I$;
\item the closure in $\RR^k$ of the set $$\{(-\val(x_1),\ldots,-\val(x_k))\,|\,(x_1,\ldots,x_k)\in V(I)\}\subset \QQ^k.$$
\end{enumerate}
\end{Thm}
\begin{proof} See \cite[Theorem 2.1]{SS1}.
\end{proof}

For an ideal $I\subset K[x_1,\ldots,x_k]$, the set mentioned in Theorem \ref{Thm def tropical variety} is called the {\it tropical variety} $\TT(I)\subset \RR^k$ of the ideal $I$.

We say that $\{f_1,\ldots,f_r\}$ is a tropical basis of $\TT(I)$ if and only if $I=\langle f_1,\ldots,f_r\rangle$ and $$\TT(I)=\TT(\trop(f_1))\cap \cdots \cap \TT(\trop(f_r)).$$

We are particularly interested in {\it tropical Grassmannians} $\mathcal{G}_{m,n}=\mathcal{T}(I_{m,n})$. In this case, the ideal $$I_{m,n}\subset K[x_{i_1\ldots i_m}|1\leq i_1<\ldots<i_m\leq n]$$ is the ideal of the affine Grassmannian $G(m,n)\subset K^{n\choose m}$ parameterizing linear subspaces of dimension $m$ in $K^n$. The ideal $I_{m,n}$ consists of all relations between the $(m\times m)$-minors of an $(m\times n)$-matrix.

\begin{Rmk} {\fontshape{n}\selectfont
In case $m=2$, the {\it Pl\"ucker relations} $$p_{ijkl}:=x_{ij}x_{kl}-x_{ik}x_{jl}+x_{il}x_{jk}$$ (with $i<j<k<l$) generate the ideal $I_{2,n}$. One can show that these polynomials also form a tropical basis of $I_{2,n}$, hence $\mathcal{G}_{2,n}$ is the intersection of the tropical hypersurfaces $\mathcal{H}(\trop(p_{ijkl}))$. Note that $\trop(p_{ijkl})$ is equal to $$(x_{ij}\otimes x_{kl})\oplus (x_{ik}\otimes x_{jl})\oplus (x_{il}\otimes x_{jk})=\max\{x_{ij}+x_{kl},x_{ik}+x_{jl},x_{il}+x_{jk}\},$$ so we get Theorem \ref{thm space trees} using Theorem \ref{thm tree metric}.
} \end{Rmk}

\section{The case $m=4$ : the proof of the main theorem}

\begin{Rmk} \label{rmk phi4} {\fontshape{n}\selectfont
Let $\phi^{(4)}:\RR^{n \choose 2}\to\RR^{n\choose 4}$ be the map sending $X=(X(i,j))_{i<j}$ to $(X(i,j,k,l))_{i<j<k<l}$, where $X(i,j,k,l)$ is the minimum of the three terms \begin{align*} X(i,j)+X(j,k)+X(k,l)+X(i,l), \\X(i,j)+X(j,l)+X(k,l)+X(i,k),\\ X(i,k)+X(j,k)+X(j,l)+X(i,l),\end{align*} divided by two. By Theorem \ref{thm phi}, the map $\phi^{(4)}$ sends the dissimilarity vector $D$ of a tree $T$ to its $4$-dissimilarity vector $(D(i,j,k,l))_{i<j<k<l}$.
} \end{Rmk}

We will now prove the main theorem.

\begin{proof}[Proof of Theorem \ref{thm main}]
Since the inclusion $\phi^{(4)}(\TN)\subset \phi^{(4)}(\RR^{n\choose 2})$ is evident, we only have to prove $\phi^{(4)}(\TN)\subset \mathcal{G}_{4,n}$.

Let $T$ be a tree with $4$-dissimilarity vector $$\mathcal{D}:=(D(i,j,k,l))_{i<j<k<l}=\phi^{(4)}((D(i,j))_{i<j})\in \phi^{(4)}(\TN)\subset\RR^{n\choose 2}.$$ If $M\in K^{4\times n}$, we denote by $M(i,j,k,l)$ the $4\times 4$-minor coming from the columns $i,j,k,l$ of $M$. The tropical Grassmannian is the closure in $\RR^{n\choose 4}$ of the set $$S:=\{(-\val(\det(M(i,j,k,l))))_{i<j<k<l}\,|\,M\in K^{4\times n}\}\subset \QQ^{{n\choose 4}}.$$

Assume first that all edges of $T$ have rational length, hence $\mathcal{D}\in\QQ^{n \choose 4}$. We are going to show that $\mathcal{D}\in S$.

Fix a rational number $E$ with $E\geq D(i,n)$ for all $i$. Define a new metric $D'$ by $$D'(i,j)=2E+D(i,j)-D(i,n)-D(j,n)$$ for all different $i,j\in[n]$, in particular $D'(i,n)=2E$ for $i\neq n$. Note that $D'\in \TN$ and that $D'$ an ultrametric on $\{1,\ldots,n-1\}$, so it can be realized by an equidistant $(n-1)$-tree $T''$ with root $r$. Each edge $e$ of $T''$ has a well-defined height $h(e)$, which is the distance from the top node of $e$ to each leaf below $e$. Pick random rational numbers $a(e)$ and $b(e)$ for every edge $e$ of $T''$. If $i\in\{1,\ldots,n-1\}$ is a leaf of $T''$, define the polynomial $x_i(t)$ resp. $y_i(t)$ as the sum of the monomials $a(e)t^{2h(e)}$ resp. $b(e)t^{2h(e)}$, where $e$ is an edge between $r$ and $i$. It is easy to see that $$D'(i,j)=\deg(x_j(t)-x_i(t))=\deg(y_j(t)-y_i(t))$$ for all $i,j\in\{1,\ldots,n-1\}$.

Denote the distance from $r$ to each leaf by $F$. Since $$2F=\max\{D'(i,j)\,|\,1\leq i< j\leq n-1\}<2E,$$ we have $F<E$. The metric $D'$ on $[n]$ can be realized by a tree $T'$, where $T'$ is the tree obtained from $T''$ by adding the leaf $n$ together with an edge $(r,n)$ of length $2E-F$.
If we define $x_n(t)=y_n(t)=t^{2E}$, we get that $D'(i,j)=\deg(x_j(t)-x_i(t))=\deg(y_j(t)-y_i(t))$ for all $i,j\in[n]$.

Consider the matrix $$M':=\begin{bmatrix} 1&1&1&1&\ldots&1\\ x_1(t)&x_2(t)&x_3(t)&x_4(t)&\ldots&x_n(t)\\ x_1(t)^2&x_2(t)^2&x_3(t)^2&x_4(t)^2&\ldots&x_n(t)^2\\ y_1(t)&y_2(t)&y_3(t)&y_4(t)&\ldots&y_n(t) \end{bmatrix}.$$

We claim that $\deg(\det(M'(i,j,k,l)))=2 D'(i,j,k,l)$ for all $i,j,k,l\in[n]$. After renumbering the leaves, we may assume that $\{i,j,k,l\}=\{1,2,3,4\}$ and that $D'(1,2)\leq D'(1,3)\leq D'(1,4)$. In Figure \ref{types}, all combinatorial types of the subtrees are pictured. Every edge in this picture may consist of several edges of the tree $T'$. Note that types I and II are different, since the top node $v$ sits on a different edge of the subtree. The type III case is special, since $n\in\{i,j,k,l\}$ (before the renumbering).

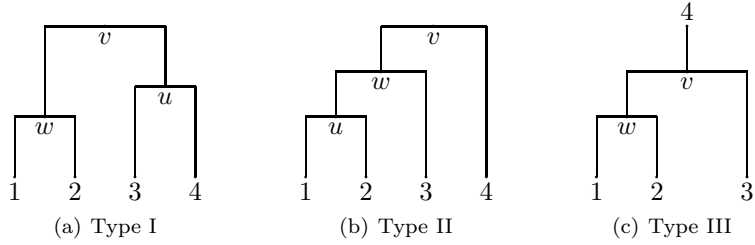
\begin{figure}[h]
\unitlength=0.4mm
\centering
\subfigure[Type I]{
\begin{picture}(60,60)(0,-8)
\put(0,0){\line(0,1){20}}
\put(20,0){\line(0,1){20}}
\put(40,0){\line(0,1){30}}
\put(60,0){\line(0,1){30}}
\put(0,20){\line(1,0){20}}
\put(40,30){\line(1,0){20}}
\put(10,20){\line(0,1){30}}
\put(50,30){\line(0,1){20}}
\put(10,50){\line(1,0){40}}
\put(0,0){\circle*{1}}
\put(0,-2){\makebox(0,0)[t]{$1$}}
\put(20,0){\circle*{1}}
\put(20,-2){\makebox(0,0)[t]{$2$}}
\put(40,0){\circle*{1}}
\put(40,-2){\makebox(0,0)[t]{$3$}}
\put(60,0){\circle*{1}}
\put(60,-2){\makebox(0,0)[t]{$4$}}
\put(10,20){\circle*{1}}
\put(10,18){\makebox(0,0)[t]{$w$}}
\put(30,50){\circle*{1}}
\put(30,48){\makebox(0,0)[t]{$v$}}
\put(50,30){\circle*{1}}
\put(50,28){\makebox(0,0)[t]{$u$}}
\end{picture}
}
\hspace{10mm}
\subfigure[Type II]{
\begin{picture}(60,60)(0,-8)
\put(0,0){\line(0,1){20}}
\put(20,0){\line(0,1){20}}
\put(40,0){\line(0,1){35}}
\put(60,0){\line(0,1){50}}
\put(0,20){\line(1,0){20}}
\put(10,35){\line(1,0){30}}
\put(25,50){\line(1,0){35}}
\put(10,20){\line(0,1){15}}
\put(25,35){\line(0,1){15}}
\put(0,0){\circle*{1}}
\put(0,-2){\makebox(0,0)[t]{$1$}}
\put(20,0){\circle*{1}}
\put(20,-2){\makebox(0,0)[t]{$2$}}
\put(40,0){\circle*{1}}
\put(40,-2){\makebox(0,0)[t]{$3$}}
\put(60,0){\circle*{1}}
\put(60,-2){\makebox(0,0)[t]{$4$}}
\put(25,35){\circle*{1}}
\put(25,33){\makebox(0,0)[t]{$w$}}
\put(42.5,50){\circle*{1}}
\put(42.5,48){\makebox(0,0)[t]{$v$}}
\put(10,20){\circle*{1}}
\put(10,18){\makebox(0,0)[t]{$u$}}
\end{picture}
}
\hspace{10mm}
\subfigure[Type III]{
\begin{picture}(50,60)(0,-8)
\put(0,0){\line(0,1){20}}
\put(20,0){\line(0,1){20}}
\put(50,0){\line(0,1){35}}
\put(10,20){\line(0,1){15}}
\put(30,35){\line(0,1){15}}
\put(0,20){\line(1,0){20}}
\put(10,35){\line(1,0){40}}
\put(0,0){\circle*{1}}
\put(0,-2){\makebox(0,0)[t]{$1$}}
\put(20,0){\circle*{1}}
\put(20,-2){\makebox(0,0)[t]{$2$}}
\put(50,0){\circle*{1}}
\put(50,-2){\makebox(0,0)[t]{$3$}}
\put(30,50){\circle*{1}}
\put(30,52){\makebox(0,0)[b]{$4$}}
\put(10,20){\circle*{1}}
\put(10,18){\makebox(0,0)[t]{$w$}}
\put(30,35){\circle*{1}}
\put(30,33){\makebox(0,0)[t]{$v$}}
\end{picture}
}
\caption{The combinatorial types of $4$-subtrees}
\label{types}
\end{figure}

The determinant of $M'(1,2,3,4)$ is equal to
\begin{eqnarray} \begin{vmatrix} 1&1&1&1\\ x_1&x_2&x_3&x_4 \\ x_1^2&x_2^2&x_3^2&x_4^2\\ y_1&y_2&y_3&y_4\end{vmatrix} &=& \begin{vmatrix}1&1&1&1\\ x_1&x_2&x_3&x_4\\ x_1^2&x_2^2&x_3^2&x_4^2\\ 0&y_2-y_1&y_3-y_1&y_4-y_1\end{vmatrix} \nonumber \\ &=& (y_2-y_1)(x_4-x_1)(x_3-x_1)(x_4-x_3) \nonumber \\ &&-(y_3-y_1)(x_4-x_1)(x_2-x_1)(x_4-x_2) \nonumber \\ &&+(y_4-y_1)(x_3-x_1)(x_2-x_1)(x_3-x_2) \label{determinant}
\end{eqnarray}
The degree of the term $(y_2-y_1)(x_4-x_1)(x_3-x_1)(x_4-x_3)$ in \eqref{determinant} is $$D'(1,2)+D'(1,4)+D'(1,3)+D'(3,4),$$ which equals $2 D'(1,2,3,4)$ for each of the three types.

If $v$ and $w$ are nodes between $r$ and $i$, we will denote the sum of the monomials $a(e)t^{2h(e)}$ for $e$ between $v$ and $w$ by $x_{i,[v,w]}(t)$. Analogously, we define $y_{i,[v,w]}(t)$. \\

We are going to take a look at the type I case. In Figure \ref{typeI}, the arrows stand for edges of $T'$. For example, the edge $e_v$ is adjacent to $v$ and goes into the direction of $w$.

\begin{figure}[h]
\unitlength=.6mm
\begin{center}
\begin{picture}(60,60)
\put(0,0){\line(0,1){20}}
\put(20,0){\line(0,1){20}}
\put(40,0){\line(0,1){30}}
\put(60,0){\line(0,1){30}}
\put(0,20){\line(1,0){20}}
\put(40,30){\line(1,0){20}}
\put(10,20){\line(0,1){30}}
\put(50,30){\line(0,1){20}}
\put(10,50){\line(1,0){40}}
\put(0,0){\circle*{1}}
\put(0,-2){\makebox(0,0)[t]{$1$}}
\put(20,0){\circle*{1}}
\put(20,-2){\makebox(0,0)[t]{$2$}}
\put(40,0){\circle*{1}}
\put(40,-2){\makebox(0,0)[t]{$3$}}
\put(60,0){\circle*{1}}
\put(60,-2){\makebox(0,0)[t]{$4$}}
\put(10,20){\circle*{1}}
\put(10,18){\makebox(0,0)[t]{$w$}}
\put(30,50){\circle*{1}}
\put(30,48){\makebox(0,0)[t]{$v$}}
\put(50,30){\circle*{1}}
\put(50,28){\makebox(0,0)[t]{$u$}}
\put(8,22){\vector(-1,0){9}}
\put(5,24){\makebox(0,0)[b]{$e_{w}$}}
\put(12,22){\vector(1,0){9}}
\put(15,24){\makebox(0,0)[b]{$e'_{w}$}}
\put(48,32){\vector(-1,0){9}}
\put(45,34){\makebox(0,0)[b]{$e_{u}$}}
\put(52,32){\vector(1,0){9}}
\put(55,34){\makebox(0,0)[b]{$e'_{u}$}}
\put(28,52){\vector(-1,0){9}}
\put(25,54){\makebox(0,0)[b]{$e_{v}$}}
\put(32,52){\vector(1,0){9}}
\put(35,54){\makebox(0,0)[b]{$e'_{v}$}}
\end{picture}
\end{center}
\caption{Type I} \label{typeI}
\end{figure}

Denote $x:=x_{3,[v,u]}-x_{1,[v,w]}$, $x_{12}:=x_{2,[w,2]}-x_{1,[w,1]}$, $x_{13}:=x_{3,[u,3]}-x_{1,[w,1]}$, etc. Analogously, we define $y,y_{12},y_{13},\ldots,y_{34}$. The determinant \eqref{determinant} equals
\begin{multline} y_{12}x_{34}(x+x_{13})(x+x_{14}) - x_{12}(y+y_{13})(x+x_{14})(x+x_{24}) \\ + x_{12}(y+y_{14})(x+x_{13})(x+x_{23}). \label{determinantI}\end{multline}
Since $\deg(x)=\deg(y)$ is bigger than $\deg(x_{ij})=\deg(y_{ij})$ for all $i$ and $j$, we have that the degree of the last two terms is equal to $$\deg(x_{12}yx^2)>2 D'(1,2,3,4),$$ but the term $x_{12}yx^2$ vanishes in the determinant. So, the degree of the sum of the last two terms in \eqref{determinantI} is equal to
\begin{align*}
&\deg[x_{12}(x^2(y_{14}-y_{13})+ xy(x_{13}+x_{23}-x_{14}-x_{24}))] \\ &=\deg[x_{12}(y_{34}x^2-2x_{34}xy)] \\ &=2 D'(1,2,3,4).
\end{align*}
We conclude that the determinant of $M'(1,2,3,4)$ has degree $2 D'(1,2,3,4)$. Indeed, the coefficient of $t^{2 D'(1,2,3,4)}$ is equal to
\begin{align*}
&(b(e'_w)-b(e_w))(a(e'_u)-a(e_u))(a(e'_v)-a(e_v))^2 \\ &+(b'(e_u)-b(e_u))(a(e'_w)-a(e_w))(a(e'_v)-a(e_v))^2 \\ &-2(b(e'_v)-b(e_v))(a(e'_v)-a(e_v))(a(e'_w)-a(e_w))(a(e'_u)-a(e_u)) \neq 0.
\end{align*}

For type II and III, the first two terms in \eqref{determinant} have degree $2 D'(1,2,3,4)$ and the last term has a lower degree. Using the notation in Figure \ref{typeIIandIII}, the coefficient of $t^{2 D'(1,2,3,4)}$ in $\det(M'(1,2,3,4))$ is equal to
$$(a(e'_v)-a(e_v))^2[(b(e'_u)-b(e_u))(a(e'_w)-a(e_w))-(b(e'_w)-b(e_w))(a(e'_u)-a(e_u))]\neq 0$$ for type II and $$(b(e'_u)-b(e_u))(a(e'_w)-a(e_w))-(b(e'_w)-b(e_w))(a(e'_u)-a(e_u))\neq 0$$ for type III. \\

\begin{figure}[h]
\unitlength=0.6mm
\centering
\subfigure{
\begin{picture}(60,60)(0,-8)
\put(0,0){\line(0,1){20}}
\put(20,0){\line(0,1){20}}
\put(40,0){\line(0,1){35}}
\put(60,0){\line(0,1){50}}
\put(0,20){\line(1,0){20}}
\put(10,35){\line(1,0){30}}
\put(25,50){\line(1,0){35}}
\put(10,20){\line(0,1){15}}
\put(25,35){\line(0,1){15}}
\put(0,0){\circle*{1}}
\put(0,-2){\makebox(0,0)[t]{$1$}}
\put(20,0){\circle*{1}}
\put(20,-2){\makebox(0,0)[t]{$2$}}
\put(40,0){\circle*{1}}
\put(40,-2){\makebox(0,0)[t]{$3$}}
\put(60,0){\circle*{1}}
\put(60,-2){\makebox(0,0)[t]{$4$}}
\put(25,35){\circle*{1}}
\put(25,33){\makebox(0,0)[t]{$w$}}
\put(42.5,50){\circle*{1}}
\put(42.5,48){\makebox(0,0)[t]{$v$}}
\put(10,20){\circle*{1}}
\put(10,18){\makebox(0,0)[t]{$u$}}
\put(8,22){\vector(-1,0){9}}
\put(5,24){\makebox(0,0)[b]{$e_{u}$}}
\put(12,22){\vector(1,0){9}}
\put(15,24){\makebox(0,0)[b]{$e'_{u}$}}
\put(23,37){\vector(-1,0){9}}
\put(20,39){\makebox(0,0)[b]{$e_{w}$}}
\put(27,37){\vector(1,0){9}}
\put(30,39){\makebox(0,0)[b]{$e'_{w}$}}
\put(40.5,52){\vector(-1,0){9}}
\put(37.5,54){\makebox(0,0)[b]{$e_{v}$}}
\put(44.5,52){\vector(1,0){9}}
\put(47.5,54){\makebox(0,0)[b]{$e'_{v}$}}
\end{picture}
}
\hspace{10mm}
\subfigure{
\begin{picture}(50,60)(0,-8)
\put(0,0){\line(0,1){20}}
\put(20,0){\line(0,1){20}}
\put(50,0){\line(0,1){35}}
\put(10,20){\line(0,1){15}}
\put(30,35){\line(0,1){15}}
\put(0,20){\line(1,0){20}}
\put(10,35){\line(1,0){40}}
\put(0,0){\circle*{1}}
\put(0,-2){\makebox(0,0)[t]{$1$}}
\put(20,0){\circle*{1}}
\put(20,-2){\makebox(0,0)[t]{$2$}}
\put(50,0){\circle*{1}}
\put(50,-2){\makebox(0,0)[t]{$3$}}
\put(30,50){\circle*{1}}
\put(30,52){\makebox(0,0)[b]{$4$}}
\put(10,20){\circle*{1}}
\put(10,18){\makebox(0,0)[t]{$w$}}
\put(30,35){\circle*{1}}
\put(30,33){\makebox(0,0)[t]{$v$}}
\put(8,22){\vector(-1,0){9}}
\put(5,24){\makebox(0,0)[b]{$e_{w}$}}
\put(12,22){\vector(1,0){9}}
\put(15,24){\makebox(0,0)[b]{$e'_{w}$}}
\put(28,37){\vector(-1,0){9}}
\put(25,39){\makebox(0,0)[b]{$e_{v}$}}
\put(32,37){\vector(1,0){9}}
\put(35,39){\makebox(0,0)[b]{$e'_{v}$}}
\end{picture}
}
\caption{Type II and III}
\label{typeIIandIII}
\end{figure}
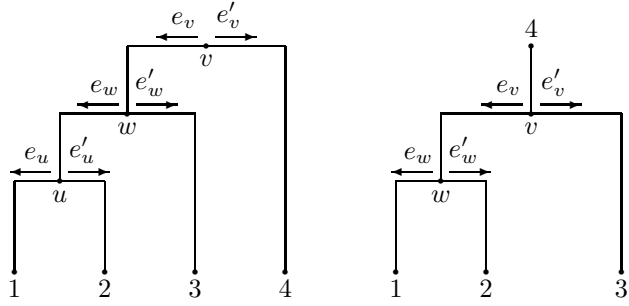

Let $M$ be the matrix obtained from $M'$ by multiplying, for each $i$, the $i$-th column of $M'$ by $(t^{D(i,n)-E})^2$. We have
\begin{eqnarray*} D(i,j) &=&D'(i,j)+(D(i,n)-E)+(D(j,n)-E)\\ &=&\deg\left(t^{D(i,n)-E} \cdot t^{D(j,n)-E} \cdot (x_i(t)-x_j(t))\right).\end{eqnarray*}
Using Remark \ref{rmk phi4}, we get that $2 D(i,j,k,l)=\deg(\det(M(i,j,k,l)))$. If we replace each $t$ in $M$ by $t^{-1/2}$, we have $$D(i,j,k,l)=-\val(\det(M(i,j,k,l))),$$ hence $\mathcal{D}\in S$.\\

Now assume $T$ has irrational edge weights. We can approximate $T$ arbitrarily close by a tree $\widetilde{T}$ with rational edge weights. From the arguments above, it follows that the $4$-dissimilarity vector $\widetilde{\mathcal{D}}$ of $\widetilde{T}$ belongs to $S$, hence $\mathcal{D}\in \mathcal{G}_{4,n}$.
\end{proof}

\section{What about the case $m\geq 5$?}

The proof of Theorem \ref{thm main} does not give an obstruction for the following to be true for $m\geq 5$.

\begin{Conj}\label{conj1}
$\phi^{(m)}(\TN)\subset \mathcal{G}_{m,n}\cap \phi^{(m)}(\RR^{n\choose 2})$
\end{Conj}

Note that using the same arguments as in the proof of Theorem \ref{thm main}, it suffices to show the following.

\begin{Conj}\label{conj2} Let $m\leq n$ be integers and let $T'$ be a weighted equidistant $(n-1)$-tree with root $r$ such that all edges of $T'$ have rational length. Denote the distance between $r$ and each leaf of $T'$ by $d'$.

Let $T$ be the tree attained from $T'$ by adding an edge $(r,n)$ of length $d''\in\QQ$ with $d''>d'$.

For each edge $e$ of $T'$, pick random numbers $a_1(e),\ldots,a_{m-2}(e)\in\mathbb{C}$ and denote its height in $T'$ by $h(e)$. Let $x_i^{(j)}(t)\in K$ (with $i\in\{1,\ldots,n-1\}$ and $j\in\{1,\ldots,m-2\}$) be the sum of the monomials $a_j(e)t^{h(e)}$, where $e$ runs over all edges between $r$ and $i$, and define $$x_n^{(1)}(t)=\ldots=x_n^{(m-2)}(t)=t^{(d'+d'')/2}\in K.$$ Consider the matrix $$
M=\begin{bmatrix} 1&1&\ldots&1\\ x_1^{(1)}&x_2^{(1)}&\ldots&x_n^{(1)}\\ (x_1^{(1)})^2&(x_2^{(1)})^2&\ldots&(x_n^{(1)})^2\\ x_1^{(2)}&x_2^{(2)}&\ldots&x_n^{(2)} \\ \vdots&\vdots&\vdots&\vdots \\ x_1^{(m-2)}&x_2^{(m-2)}&\ldots&x_n^{(m-2)}\end{bmatrix} \in K^{m\times n}.
$$

Let $i_1,\ldots,i_m$ be pairwise disjoint elements in $\{1,\ldots,n\}$. Then we have that $D(i_1,\ldots,i_m)=\deg(\det(M(i_1,\ldots,i_m)))$.
\end{Conj}

\begin{Rmk} {\fontshape{n}\selectfont
The matrix $M$ arising in Conjecture \ref{conj1} has a sort of asymmetry. However, if one would construct polynomials $x_i^{(j)}$ as in the conjecture with $j\in \{1,\ldots,m\}$ for each leaf $i\in\{1,\ldots,n\}$, the statement fails for $$N=\begin{bmatrix} x_1^{(1)}&x_2^{(1)}&\ldots&x_n^{(1)}\\ x_1^{(2)}&x_2^{(2)}&\ldots&x_n^{(2)} \\ \vdots&\vdots&\vdots&\vdots \\ x_1^{(m)}&x_2^{(m)}&\ldots&x_n^{(m)}\end{bmatrix} \in K^{m\times n},$$ even for $m=3$. Indeed, if the minimal subtree $\widetilde{T}$ of the equidistant tree $T'$ containing the three leaves $i_1,i_2,i_3$ does not contain the root $r$, the degree of the determinant of $N(i_1,i_2,i_3)$ is not equal to the length of $\widetilde{T}$. Instead, it is equal to the length of the subtree of $T'$ containing the leaves $i_1,i_2,i_3$ and the root $r$. The same happens for $m=4$. So it seems that the row consisting of ones in the matrix $M$ is necessary to cancel the distance between the top node of $\widetilde{T}$ and the root $r$. On the other hand, the determinant of a maximal minor has to be homogeneous in the variables $x_i^{(j)}$ of degree $m$ (see Theorem \ref{thm phi}), so once we put a row with ones in $M$, there should be a row consisting of quadric forms in the variables $x_i^{(j)}$, i.e. the third row of $M$. 
} \end{Rmk}

We can simplify Conjecture \ref{conj2}. Firstly, we can see that the tree $T$ can be considered as an equidistant $n$-tree, if we pick the top node to be the node on the edge $(r,n)$ at distance $(d'+d'')/2$ of $n$. For example, in the proof of Theorem \ref{thm main}, the types II and III are in fact equivalent. Secondly, assume $I=\{i_1,\ldots,i_m\}$ is an $m$-subset of $\{1,\ldots,n\}$ and let $T_I$ be the minimal subtree of $T$ containing the leafs in $I$. The edges between the top node $r_I$ of $T_I$ and the root $r$ of $T$ do not give a contribution in the determinant of $M(I)=M(i_1,\ldots,i_m)$. Also, the edges of $T_I$ with $2$-valent top node different from $r_I$ can be canceled out in the computation of $\deg(\det(M(I)))$. So we see that Conjecture \ref{conj2} is equivalent to the following.

\begin{Conj} \label{conj3}
Let $T$ be an equidistant $m$-tree with root $r$ such that all edges of $T$ have rational length.

For each edge $e$ of $T$, pick random numbers $a_1(e),\ldots,a_{m-2}(e)\in\mathbb{C}$ and denote its height in $T$ by $h(e)$. Let $x_i^{(j)}(t)\in K$ (with $i\in\{1,\ldots,m\}$ and $j\in\{1,\ldots,m-2\}$) be the sum of the monomials $a_j(e)t^{h(e)}$, where $e$ runs over all edges between $r$ and $i$.
Then the degree of the determinant of $$M=\begin{bmatrix} 1&1&\ldots&1\\ x_1^{(1)}&x_2^{(1)}&\ldots&x_m^{(1)}\\ (x_1^{(1)})^2&(x_2^{(1)})^2&\ldots&(x_m^{(1)})^2\\ x_1^{(2)}&x_2^{(2)}&\ldots&x_m^{(2)} \\ \vdots&\vdots&\vdots&\vdots \\ x_1^{(m-2)}&x_2^{(m-2)}&\ldots&x_m^{(m-2)}\end{bmatrix}$$
is equal to the length $D$ of $T$.
\end{Conj}

We give an example to illustrate Conjecture \ref{conj3} for $m=5$.

\begin{Ex} \label{ex m5} {\fontshape{n}\selectfont
Consider the equidistant $5$-tree $T$ of Figure \ref{Fig Example}. In the boxes, the distances of the edges are mentioned. Note that $D=37$.
\begin{figure}[h]
\unitlength=.6mm
\begin{center}
\begin{picture}(73,50)
\put(0,0){\line(0,1){20}}
\put(20,0){\line(0,1){20}}
\put(10,20){\line(0,1){15}}
\put(23,35){\line(0,1){15}}
\put(53,0){\line(0,1){30}}
\put(73,0){\line(0,1){30}}
\put(36,0){\line(0,1){35}}
\put(63,30){\line(0,1){20}}
\put(0,20){\line(1,0){20}}
\put(10,35){\line(1,0){26}}
\put(23,50){\line(1,0){40}}
\put(53,30){\line(1,0){20}}
\put(0,0){\circle*{1}}
\put(0,-2){\makebox(0,0)[t]{$1$}}
\put(20,0){\circle*{1}}
\put(20,-2){\makebox(0,0)[t]{$2$}}
\put(36,0){\circle*{1}}
\put(36,-2){\makebox(0,0)[t]{$3$}}
\put(53,0){\circle*{1}}
\put(53,-2){\makebox(0,0)[t]{$4$}}
\put(73,0){\circle*{1}}
\put(73,-2){\makebox(0,0)[t]{$5$}}
\put(10,20){\circle*{1}}
\put(10,18){\makebox(0,0)[t]{$w$}}
\put(23,35){\circle*{1}}
\put(23,33){\makebox(0,0)[t]{$v$}}
\put(43,50){\circle*{1}}
\put(43,48){\makebox(0,0)[t]{$r$}}
\put(63,30){\circle*{1}}
\put(63,28){\makebox(0,0)[t]{$u$}}
\put(-3.5,13){\makebox(0,0)[t]{$\small{\fbox{\hspace{-0.7mm}4\hspace{-0.6mm}}}$}}
\put(16.5,13){\makebox(0,0)[t]{$\small{\fbox{\hspace{-0.7mm}4\hspace{-0.6mm}}}$}}
\put(6.5,31){\makebox(0,0)[t]{$\small{\fbox{\hspace{-0.7mm}3\hspace{-0.6mm}}}$}}
\put(19.5,46){\makebox(0,0)[t]{$\small{\fbox{\hspace{-0.7mm}3\hspace{-0.6mm}}}$}}
\put(66.5,43){\makebox(0,0)[t]{$\small{\fbox{\hspace{-0.7mm}4\hspace{-0.6mm}}}$}}
\put(39.5,20){\makebox(0,0)[t]{$\small{\fbox{\hspace{-0.7mm}7\hspace{-0.6mm}}}$}}
\put(56.5,18){\makebox(0,0)[t]{$\small{\fbox{\hspace{-0.7mm}6\hspace{-0.6mm}}}$}}
\put(76.5,18){\makebox(0,0)[t]{$\small{\fbox{\hspace{-0.7mm}6\hspace{-0.6mm}}}$}}
\end{picture}
\end{center}
\caption{Equidistant $5$-tree $T$} \label{Fig Example}
\end{figure}
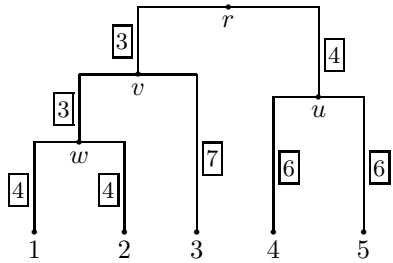

Following the notations of Conjecture \ref{conj3}, we have
\begin{eqnarray*}
x_1^{(j)}(t)&=&a_j(r,v)\,t^{10}+a_j(v,w)\,t^7+a_j(w,1)\,t^4 , \\
x_2^{(j)}(t)&=&a_j(r,v)\,t^{10}+a_j(v,w)\,t^7+a_j(w,2)\,t^4 ,\\
x_3^{(j)}(t)&=&a_j(r,v)\,t^{10}+a_j(v,3)\,t^7 ,\\
x_4^{(j)}(t)&=&a_j(r,u)\,t^{10}+a_j(u,4)\,t^6 ,\\
x_5^{(j)}(t)&=&a_j(r,u)\,t^{10}+a_j(u,5)\,t^6 .
\end{eqnarray*}
Using a computer algebra system, one can see that the determinant of $M$
is a polynomial of degree $37$ in the variable $t$. Each of its coefficients is homogeneous of degree $5$ in the numbers $a_j(e)$, with $j\in \{1,2,3\}$ and $e$ an edge of $T$.

If we take the numbers $a_j(e)$ to be the first $24=3\times 8$ prime numbers (i.e. $a_1(r,v)=2, \ldots, a_3(u,5)=89$),
the determinant of $M$ has leading coefficient $3344$.
} \end{Ex}

\begin{Rmk} {\fontshape{n}\selectfont
In order to prove Conjecture \ref{conj3} for a fixed value of $m$, one could follow the strategy of Theorem \ref{thm main}. Indeed, the number $t(m)$ of combinatorial types of equidistant $m$-trees is finite and for each of these types, one can compute the determinant of $M$ and check whether its degree equals $D$. 

In this way, we can prove Conjecture \ref{conj3} for $m=5$ using a computer algebra system. For each of the three combinatorial types of equidistant $5$-trees, the determinant of $M$ can be computed, leaving the random numbers $a_j(e)$ and the lengths $l(e)$ of the edges as variables. This determinant (considered as a polynomial in the variable $t$) has degree equal to the length $D$ of the tree $T$ and its leading coefficient is a homogeneous polynomial $c_T$ of degree $5$ in the numbers $a_j(e)$. If the tree $T$ is binary, the polynomial $c_T$ has $272$ terms for the type corresponding to Example \ref{ex m5}, and $144$ terms for the other two types. Note that the numbers $a_j(e)$ are sufficiently random if they don't vanish for the polynomial $c_T$. We can conclude that the inclusion $$\phi^{(5)}(\GT)\subset \mathcal{G}_{5,n}\cap \phi^{(5)}(\RR^{n\choose 2})$$ holds, i.e. Conjecture \ref{conj1} for $m=5$. 

On the other hand, the number $t(m)$ grows exponentially, e.g. $$t(4)=2, t(5)=3, t(6)=6, t(7)=11, t(8)=23, t(9)=46, t(10)=98, etc.,$$ and for each of these types, the square matrix $M$ is of size $m$, hence the computation of its determinant gets more complicated when $m$ grows. So this technique is not suited in order to prove Conjecture \ref{conj3} for every $m$. However, one can hope to find a proof by induction on $m$.
}\end{Rmk}

\end{document}